\documentclass [11pt]{amsart}  
\usepackage{a4}                
\usepackage{paralist}           
\usepackage{graphicx}
\usepackage{amssymb}                        
\usepackage{float}
\usepackage{amsmath}
\usepackage{psfrag}
\usepackage{MnSymbol}
\usepackage[all]{xy}

\usepackage{amsmath}
\usepackage{wasysym}
\usepackage{latexsym}
\usepackage{amsfonts}
\usepackage{mathrsfs}
\usepackage{dsfont}

\usepackage{amsmath,amsthm}

\theoremstyle{plain} 
\newtheorem{thm}{Theorem}[section]

\newtheorem{lem}[thm]{Lemma}

\newtheorem{prop}[thm]{Proposition}
\newtheorem{cor}[thm]{Corollary}

\theoremstyle{definition}
\newtheorem{defn}[thm]{Definition}

\theoremstyle{remark}
\newtheorem{rem}[thm]{Remark}

\newtheoremstyle{TheoremNum}
        {\topsep}{\topsep}              
        {\itshape}                      
        {}                              
        {\bfseries}                     
        {.}                             
        { }                             
        {\thmname{#1}\thmnote{ \bfseries #3}}
    \theoremstyle{TheoremNum}
    \newtheorem{thmn}{Theorem}

\DeclareMathOperator{\id}{id}

\DeclareMathOperator{\tor}{Tor}

\DeclareMathOperator{\torord}{Tord}
\DeclareMathOperator{\torlen}{TorLen}

\DeclareMathOperator{\T}{T}
\DeclareMathOperator{\tf}{tf}

\DeclareMathOperator{\ord}{o}
\DeclareMathOperator{\fac}{fc}

\title{$X$-torsion and universal groups}
\author{Maurice Chiodo, Zachiri McKenzie}
\date{\today}

\begin{document}

\let\thefootnote\relax\footnotetext{2010 \textit{AMS Classification:} 20F05, 03D55, 03C98.}
\let\thefootnote\relax
\let\thefootnote\relax\footnotetext{This work is part of a project that has received funding from the European Union's Horizon 2020 research and innovation programme under the Marie Sk\l{}odowska-Curie grant agreement No.~659102.}

\begin{abstract}
For a set $X\subseteq \mathbb{N}$, we define the \emph{$X$-torsion} of a group $G$ to be all elements $g\in G$ with $g^{n}=e$ for some $n\in X$. With $X$ recursively enumerable, we give two independent proofs (group-theoretic, and model-theoretic) that there exists a \emph{universal} finitely presented $X$-torsion-free group; one which contains \emph{all} finitely presented $X$-torsion-free groups. We also show that, if $X$ is recursively enumerable, then the set of finite presentations of $X$-torsion-free groups is $\Pi_{2}^{0}$-complete in Kleene's arithmetic hierarchy.
\end{abstract}

\maketitle

\section{Introduction}

Torsion is a well-studied object in group theory. We let $\ord(g)$ denote the order of a group element $g$.  Recalling that $g\in G$ is \emph{torsion} if $1\leq \ord(g) <\omega$, we write  $\tor(G):=\{g \in G\ |\ g \textnormal{ is torsion}\}$.   So, what if we were to `restrict' the type of torsion we are looking at? Perhaps we are only concerned with $2$-torsion, or torsion elements of prime order. So, for any set $X \subseteq \mathbb{N}$, we define the \emph{$X$-torsion} of a group $G$, written $\tor^{X}(G)$, as
\[
 \tor^{X}(G):=\{g \in G \ | \ \exists n \in X \textnormal{ with } g^{n}=e\}
\]
It is clear that, for the case $X=\mathbb{N}$, we have $\tor^{\mathbb{N}}(G)= \tor(G)$ in the usual sense. For any set  ${X \subseteq \mathbb{N}}$, we define the \emph{factor completion} of $X$ to be $X^{\fac} := \{ n \in \mathbb{N} \:|\: \exists m \geq 1,$ $nm \in X \}$, and we say $X$ is \emph{factor complete} if ${X^{\fac} = X}$. We say that a group $G$ is \emph{$X$-torsion-free} if $\tor^{X}(G)=\{e\}$; equivalently, if $\torord(G)\cap X^{\fac}=\emptyset$. 
We use the following notation for the set of orders of torsion elements in a group:
\[
\torord(G):=\{n \in \mathbb{N}\ |\ \exists g \in \tor(G) \textnormal{ with } \ord(g)=n\geq 2\}
\]

One famous consequence of the Higman Embedding Theorem \cite{Hig} is the fact that there is a \emph{universal} finitely presented (f.p.)~group; that is, a finitely presented group into which all finitely presented groups embed. Recently, Belegradek in \cite{BS08} and Chiodo in \cite{Chiodo3} independently showed that there exists a universal f.p.~torsion-free group; that is, an f.p.~torsion-free group into which all f.p.~torsion-free groups embed. In this paper we generalise this result further, in the context of $X$-torsion-freeness, as follows:
\\\mbox{}

\begin{thmn}[\ref{uni tor free}]
Let $X$ be a recursively enumerable set of integers. Then there is a universal finitely presented $X$-torsion-free group $G$. That is, $G$ is $X$-torsion-free, and for any finitely presented group $H$ we have that $H \hookrightarrow G$ if (and only if) $H$ is $X$-torsion-free.
\end{thmn}
\mbox{}

There are two key steps in proving this theorem. The first is to construct a free product of all such groups, in an algorithmic way:
\\\mbox{}

\begin{thmn}[\ref{uni cbl tor free}]
Let $A$ be an r.e.~set of integers. Then there is a countably generated recursive presentation $Q$ of a group $\overline{Q}$ which is $A$-torsion-free, and contains an embedded copy of every countably generated recursively presentable $A$-torsion-free group.
\end{thmn}
\mbox{}

From this, we can construct a finitely presented example, using the following theorem of Higman:
\\\mbox{}

\begin{thmn}[\ref{tor emb}]
There is a uniform algorithm that, on input of a countably generated recursive presentation $P=\langle X | R \rangle$, constructs a finite presentation $\T(P)$ such that $\overline{P}\hookrightarrow \overline{\T(P)}$ and  $\torord(\overline{P}) =\torord(\overline{\T(P)})$, along with an explicit embedding $\overline{\phi}: \overline{P} \hookrightarrow \overline{\T(P)}$.
\end{thmn}
\mbox{}

We prove Theorem \ref{uni cbl tor free} in the following two independent ways, for reasons which we outline shortly. Firstly, in Section \ref{sec:group_proof}, we generalise the construction in \cite[Theorem 3.10]{Chiodo3} of a universal f.p.~torsion-free group, using arguments in group theory. Then, in Section \ref{sec:model_proof}, we generalise the construction in \cite[Theorem A.1]{BS08} of a universal f.p.~torsion-free group, using arguments in model theory. The results in Section \ref{sec:model_proof} provide the framework for proving that there might be other group properties for which there exist universal f.p.~examples. Our main result in Section \ref{sec:model_proof}, which we apply directly to the theory of $X$-torsion-free groups to prove Theorem \ref{uni cbl tor free}, is the following:
\\\mbox{}

\begin{thmn}[\ref{Th:GeneralUniversalGroupExistenceTheorem}]
If $T \supseteq T_{\mathrm{Grp}}$ is an r.e.~universal Horn $\mathcal{L}_{\mathrm{Grp}}$-theory then there exists a recursively presented group $G \models T$ such that every recursively presented group $H \models T$ embeds into $G$.  
\end{thmn}
\mbox{}

The reason for providing both proofs of Theorem \ref{uni cbl tor free} is twofold. On the one hand, the group-theoretic proof in Section \ref{sec:group_proof} is more direct, and gives a clear picture of \emph{why} Theorem \ref{uni cbl tor free} holds. It is explicit, and algorithmic. On the other hand, the model-theoretic proof in Section \ref{sec:model_proof} is more general, and might lead to showing that other group properties possess universal f.p.~examples. Indeed, it was only by looking at the model-theoretic arguments in \cite[Theorem A.1]{BS08} that we realised we could consider $X$-torsion as an object, and that the results of \cite{Chiodo3} could be generalised to this; we would never have made the connection otherwise.

On this, it would be interesting further work to see what other group-theoretic properties $\rho$ possess universal f.p.~examples. A potential proof technique would be as follows:
\begin{enumerate}
 \item\label{en:horn} Show that the property $\rho$ satisfies the conditions of Theorem \ref{Th:ConditionAdmitsPresentations} (that is, closed under free products, identity, and subgroups).
 \item\label{en:re.horn} Show that $\rho$ has an r.e.~universal Horn $\mathcal{L}_{\mathrm{Grp}}$-theory, thus allowing us to apply Theorem \ref{Th:GeneralUniversalGroupExistenceTheorem}.
 \item\label{en:HNN} Show that $\rho$ is possessed by finitely generated free groups, and preserved by free products and HNN extensions (which is what is needed to show that it is preserved under the Higman embedding of Theorem \ref{tor emb}).
\end{enumerate}

Of course, satisfying all of the above is quite difficult. One can easily find many group properties that satisfy (\ref{en:horn}), and some of these also satisfy (\ref{en:re.horn}) quite trivially. But there are \emph{very few} properties of groups that satisfy (\ref{en:HNN}), and therein lies the problem. 

We finish our paper with a generalisation a result from \cite{Chiodo3} on the complexity of recognising torsion-freeness, and show the following:
\\\mbox{}

\begin{thmn}[\ref{pi2}]
Let $X\subseteq \mathbb{N}_{\geq 2}$ be a non-empty r.e.~set of integers. Then the set of finite presentations of $X$-torsion-free groups is $\Pi^{0}_{2}$-complete.
\end{thmn}
\mbox{}

This is remarkable; even if $X$ is a finite set, or stronger still, a single prime (say $X=\{2\}$), the set of finite presentations of groups with no $X$-torsion will \emph{still} form a $\Pi^{0}_{2}$-complete set.

\section{$X$-torsion and universality}\label{sec:group_proof}

This section is a generalisation of the definitions and results in  \cite[Section 3]{Chiodo3} on universal torsion-free quotients and universal finitely presented torsion-free groups.

\subsection{Notation}\mbox{}

If $P$ is a group presentation, we denote by $\overline{P}$ the group presented by $P$, and $\overline{w}$ by the group element represented by the word $w$. A presentation $P=\langle X|R\rangle$ is said to be a \emph{recursive presentation} if $X$ is a finite set and $R$ is a recursive enumeration of relations; $P$ is said to be a \emph{countably generated recursive presentation} if instead $X$ is a recursive enumeration of generators. 
A group $G$ is said to be \emph{finitely} (respectively,   \emph{recursively}) \emph{presentable} if $G\cong \overline{P}$ for some finite (respectively,   recursive) presentation $P$. 
If $P,Q$ are group presentations then we denote their free product presentation by $P*Q$: this is given by taking the disjoint union of their generators and relations. If $g_{1}, \ldots, g_{n}$ are elements of a group $G$, we write $\langle   g_{1}, \ldots, g_{n} \rangle$ for the subgroup in $G$ generated by these elements and $\llangle g_{1}, \ldots, g_{n} \rrangle^{G}$ for the normal closure of these elements in $G$. Let $\omega$ denote the smallest infinite ordinal.   
Let $|X|$ denote the cardinality of a set $X$. If $X$ is a set, let $X^{-1}$ be a set of the same cardinality as and disjoint from $X$ along with a fixed bijection ${*}^{-1}: X \to X^{-1}$. Write $X^{*}$ for the set of finite words on $X \cup X^{-1}$. We will make use of $\Sigma_{n}^{0}$ sets and $\Pi_{n}^{0}$ sets; see \cite{Rogers} for an introduction to these.

\subsection{$X$-torsion}\mbox{}

If $G,H$ are groups, and $X$ a set of integers with $H$ $X$-torsion-free, a surjective homomorphism $h: G \twoheadrightarrow H$ is \emph{universal} if, for any  $X$-torsion-free $K$ and any homomorphism $f:G \to K$, there is a homomorphism $\phi:H \to K$ such that $f=\phi \circ h: G \to K$, i.e., the following diagram commutes:
\begin{displaymath}
\xymatrix{
G \ar[r]^{h} \ar[dr]_{f} & H \ar[d]^{\phi} \\
 & K
}
\end{displaymath}
Note that if $\phi$ exists then it will be unique. Indeed, if $\phi'$ also satisfies $f=\phi' \circ h$, then $\phi \circ h = \phi' \circ h$, and hence $\phi=\phi'$ as $h$ is a surjection and thus is right-cancellative. Moreover, any such $H$ is unique, up to isomorphism. Such an $H$ is called the \emph{universal $X$-torsion-free quotient} for $G$, denoted $G^{X-\tf}$. Observe that if $G$ is itself $X$-torsion-free, then $G^{X-\tf}$ exists and $G^{X-\tf} \cong G$, as the identity map $\id_{G}: G \to G$ has the universal property above.

A standard construction, showing that $G^{X-\tf}$ exists for every group $G$, is done via taking the quotient of $G$ by its \emph{$X$-torsion-free radical} $\rho^{X}(G)$, where $\rho^{X}(G)$ is the intersection of all normal subgroups $N\vartriangleleft G$ with $G/N$ $X$-torsion-free (a generalisation of the \emph{torsion-free radical}, $\rho(G)$, in \cite{BrodHow}). It follows immediately that $G/\rho^{X}(G)$ has all the properties of an $X$-torsion-free universal quotient for $G$. 

We present here an alternative construction for $G^{X-\tf}$ which, though isomorphic to $G/\rho^{X}(G)$, lends itself more easily to an effective procedure for finitely (or recursively) presented groups

\begin{defn} \label{defofXtorsion}
Given a group $G$, and a set of integers $X \subseteq \mathbb{N}$, we inductively define $\tor^{X}_{n}(G)$ as follows:
\[
\tor^{X}_{0}(G):=\{e\},
\]
\[
\tor^{X}_{n+1}(G):=\llangle\ \{g \in G\ |\ g\tor^{X}_{n}(G) \in \tor^{X} \big( G/\tor^{X}_{n}(G)\big) \}\ \rrangle ^{G},
\]
\[
\tor^{X}_{\omega}(G):=\bigcup_{n \in \mathbb{N}}\tor^{X}_{n}(G).
\]
Thus, $\tor^{X}_{i}(G)$ is the set of elements of $G$ which are annihilated upon taking $i$ successive quotients of $G$ by the normal closure of all $X$-torsion elements, and $\tor^{X}_{\omega}(G)$ is the union of all these. 
\end{defn}

\begin{lem}\label{infinity tf}
If $G$ is a group, then $G/\tor^{X}_{\omega}(G)$ is $X$-torsion-free.
\end{lem}

\begin{proof}
Suppose $g\tor^{X}_{\omega}(G) \in \tor^{X} \big(G/\tor^{X}_{\omega}(G)\big)$. Then $g^{n}\tor^{X}_{\omega}(G) =e$ in $G/\tor^{X}_{\omega}(G)$ for some $1\leq n \in X$, so $g^{n} \in \tor^{X}_{\omega}(G)$. Thus there is some $i \in \mathbb{N}$ such that $g^{n} \in \tor^{X}_{i}(G)$, and hence $g\tor^{X}_{i}(G) \in \tor^{X}\big(G/\tor^{X}_{i}(G)\big)$. Thus $g \in \tor^{X}_{i+1}(G) \subseteq \tor^{X}_{\omega}(G)$, and so $g\tor^{X}_{\omega}(G) =e$ in $G/\tor^{X}_{\omega}(G)$.
\end{proof}

\begin{prop}
If $G$ is a group, then $\rho^{X}(G) = \tor^{X}_{\omega}(G)$.
\end{prop}

\begin{proof}
Clearly $\rho^{X}(G) \subseteq \tor^{X}_{\omega}(G) $, by definition of $\rho^{X}(G)$ and the fact that $G/\tor^{X}_{\omega}(G)$ is torsion-free (Lemma \ref{infinity tf}). It remains to show that $\tor^{X}_{\omega}(G) \subseteq \rho^{X}(G)$. We proceed by contradiction, so assume $\tor^{X}_{\omega}(G) \nsubseteq \rho^{X}(G)$. Then there is some $N\vartriangleleft G$ with $G/N$ $X$-torsion-free, along with some minimal $i$ such that $\tor^{X}_{i}(G) \nsubseteq N$ (clearly, $i>0$, as $\tor^{X}_{0}(G) =\{e\}$). Then, by definition of $\tor^{X}_{i}(G)$ and the fact that $N$ is normal, there exists $e \neq g \in \tor^{X}_{i}(G)$ such that $g\tor^{X}_{i-1}(G) \in \tor^{X} \big( G/\tor^{X}_{i-1}(G)\big)$ and $g \notin N$ (or else $\tor^{X}_{i}(G) \subseteq N$). But then $g^{n} \in \tor_{i-1}(G)$ for some $1<n\in X$. Since $\tor^{X}_{i-1}(G) \subseteq N$ by minimality of $i$, we have that $gN$ is a (non-trivial) $X$-torsion element of $G/N$, contradicting $X$-torsion-freeness of $G/N$. Hence $\tor^{X}_{\omega}(G) \subseteq \rho^{X}(G)$.
\end{proof}

\begin{cor}\label{tf quot}
If $G$ is a group, then $G/\tor^{X}_{\omega}(G)\cong G^{X-\tf}$. 
\end{cor}

What follows is a standard result, which we state without proof.

\begin{lem}\label{words}
Let $P=\langle X | R \rangle$ be a countably generated recursive presentation. Then the set of words $\{w \in X^{*}|\ \overline{w} =e \textnormal{ in } \overline{P}\}$ is r.e. 
\end{lem}

\begin{lem}\label{tor re}
Let $P=\langle X | R \rangle$ be a countably generated recursive presentation, and $A$ an r.e.~set of integers. Then the set of words $\{w \in X^{*}|\ \overline{w} \in \tor^{A}(\overline{P})\}$ is r.e.
\end{lem}

\begin{proof}
Take any recursive enumeration $\{w_{1}, w_{2}, \ldots \}$ of $X^{*}$. Using Lemma \ref{words}, start checking if $\overline{w}_{i}^{n}=e$ in $\overline{P}$ for each $w_{i} \in X^{*}$ and each $n \in A$ (by proceeding along finite diagonals). For each $w_{i}$ we come across which is $A$-torsion, add it to our enumeration. This procedure will enumerate all words in $\tor^{A}(\overline{P})$, and only words in $\tor^{A}(\overline{P})$. Thus the set of words in $X^{*}$ representing elements in $\tor^{A}(\overline{P})$ is r.e.
\end{proof}

We use this to show the following:

\begin{lem}\label{tor i re}\label{enum Tor i}
Given a countably generated recursive presentation $P=\langle X | R \rangle$, and r.e.~set of integers $A$, the set $T^{A}_{i}:=\{w \in X^{*}|\ \overline{w}\in \tor^{A}_{i}(\overline{P}) \}$ is r.e., uniformly over all $i$ and all such presentations $P$. Moreover, the union $T^{A}_{\omega}:=\bigcup T^{A}_{i}$ is r.e., and is precisely the set $\{w \in X^{*}|\ \overline{w} \in \tor^{A}_{\omega}(\overline{P}) \}$.
\end{lem}

\begin{proof}
We proceed by induction. Clearly $\tor^{A}_{1}(\overline{P})$ is r.e., as it is the normal closure of $\tor^{A}(\overline{P})$, which is r.e.~by Lemma \ref{tor re}. So assume that $\tor^{A}_{i}(\overline{P})$ is r.e.~for all $i \leq n$. Then $\tor^{A}_{n+1}(\overline{P})$ is the normal closure of $\tor^{A}(\overline{P}/\tor^{A}_{n}(\overline{P}))$, which again is r.e.~by the induction hypothesis and Lemma \ref{tor re}. The rest of the lemma then follows immediately.
\end{proof}

\begin{prop}\label{tkill}
There is a uniform algorithm that, on input of a countably generated recursive presentation $P=\langle X | R \rangle$ of a group $\overline{P}$, and an r.e.~set of integers $A$,  outputs a countably generated recursive presentation $P^{A-\tf}=\langle X | R' \rangle$ (on the same generating set $X$, and with $R\subseteq R'$ as sets) such that $\overline{P^{A-\tf}} \cong \overline{P}^{A-\tf}$, with associated surjection given by extending $\id_{X}: X \to X$.
\end{prop}

\begin{proof}
By Corollary \ref{tf quot}, $\overline{P}^{A-\tf}$ is the group $\overline{P}/\tor^{A}_{\omega}(\overline{P})$. Then, with the notation of Lemma \ref{tor i re}, it can be seen that $P^{A-\tf}:=\langle X | R \cup T_{\omega} \rangle$ is a countably generated recursive presentation for $\overline{P}^{A-\tf}$, uniformly constructed from $P$.
\end{proof}

\subsection{Universality and complexity of $X$-torsion}\mbox{}

With the above machinery, we can now prove the main technical result of this section; Theorem \ref{uni cbl tor free}. We will re-prove this result again in Section \ref{sec:model_proof}, using tools from model theory.

\begin{thm}\label{uni cbl tor free}
Let $A$ be an r.e.~set of integers. Then there is a countably generated recursive presentation $Q$ of a group $\overline{Q}$ which is $A$-torsion-free, and contains an embedded copy of every countably generated recursively presentable $A$-torsion-free group.
\end{thm}

\begin{proof}
Take an enumeration $P_{1}, P_{2}, \ldots$ of all countably generated recursive presentations of groups, and construct the countably generated recursive presentation $Q:=P_{1}^{A-\tf}* P_{2}^{A-\tf}* \ldots$; this is the countably infinite free product of the universal $A$-torsion-free quotient of all countably generated recursively presentable groups (with some repetition). 
As each $P_{i}^{A-\tf}$ is uniformly constructible from $P_{i}$ (by Proposition \ref{tkill}), we have that our construction of $Q$ is indeed effective, and hence $Q$ is a countably generated recursive presentation. 
Also, Proposition \ref{tkill} shows that $\overline{Q}$ is an $A$-torsion-free group, as we have successfully annihilated all the $A$-torsion in the free product factors, and the free product of $A$-torsion-free groups is again $A$-torsion-free. Moreover, $\overline{Q}$ contains an embedded copy of every $A$-torsion-free countably generated recursively presentable group, as the universal $A$-torsion-free quotient of an $A$-torsion-free group is itself. 
\end{proof}

As detailed in \cite[Lemma 6.9 and Theorem 6.10]{Chiodo1}, the following is implicit in Rotman's proof \cite[Theorem 12.18]{Rot} of the Higman Embedding Theorem.

\begin{thm}\label{tor emb}
There is a uniform algorithm that, on input of a countably generated recursive presentation $P=\langle X | R \rangle$, constructs a finite presentation $\T(P)$ such that $\overline{P}\hookrightarrow \overline{\T(P)}$ and  $\torord(\overline{P}) =\torord(\overline{\T(P)})$, along with an explicit embedding $\overline{\phi}: \overline{P} \hookrightarrow \overline{\T(P)}$.
\end{thm}

We can now prove our main result:

\begin{thm}\label{uni cbl tor free b}
Let $A$ be an r.e.~set of integers. Then there is a finitely  presentable group $G$ which is $A$-torsion-free, and contains an embedded copy of every countably generated recursively presentable $A$-torsion-free group.
\end{thm}

\begin{proof}
 We construct $Q$ as in Theorem \ref{uni cbl tor free}, and then use Theorem \ref{tor emb} to embed $\overline{Q}$ into a finitely presentable group $\overline{\T(Q)}$. By construction, $\torord(\overline{Q})=\torord(\overline{\T(Q)})$, so $\overline{\T(Q)}$ is $A$-torsion-free. Finally, $\overline{\T(Q)}$ has an embedded copy of every countably generated recursively presentable $A$-torsion-free group, since $\overline{Q}$ did. Taking $G$ to be $\overline{\T(Q)}$ completes the proof.
\end{proof}

As all f.p.~groups are recursively presentable, we have the following corollary:

\begin{thm}\label{uni tor free}
Let $A$ be an r.e.~set of integers. Then there is a universal finitely presented $A$-torsion-free group $G$. That is, $G$ is $A$-torsion-free, and for any finitely presented group $H$ we have that $H \hookrightarrow G$ if (and only if) $H$ is $A$-torsion-free.
\end{thm}

The following is an unexpected and very strong generalisation of \cite[Theorem  4.2]{Chiodo3}, classifying the computational complexity of recognising f.p.~$A$-torsion-free groups.

\begin{thm}\label{pi2}
Let $A\subseteq \mathbb{N}_{\geq 2}$ be a non-empty set of integers. Then the set of finite presentations of $A$-torsion-free groups is $\Pi^{0}_{2}$-hard. Moreover, if $A$ is r.e., then this set of presentations is $\Pi^{0}_{2}$-complete.
\end{thm}

\begin{proof}
This follows the proofs of \cite[Lemma 6.11]{Chiodo1} and \cite[Theorem 4.2]{Chiodo3}. First, there must be  some element $1<a \in A$. Now, given $n \in \mathbb{N}$, form the recursive presentation $P_{n}:=\langle x_{1}, x_{2}, \ldots \ | \ x_{i}^{a}=e \ \forall i \in \mathbb{N},\ x_{j}=e\ \forall j \in W_{n}\rangle$. Then form the finite presentation $Q_{n}$ using Higman's Embedding Theorem (Theorem \ref{tor emb}) so that $\torord(\overline{Q}_{n})=\torord{\overline{P}}_{n}$. Now note that $\overline{Q}_{n}$ is $A$-torsion-free $\Leftrightarrow$ $\overline{P}_{n}$ is $A$-torsion free $\Leftrightarrow$ $\overline{x}_{i}=e$ for all $i\in \mathbb{N}$ $\Leftrightarrow$ $W_{n}=\mathbb{N}$; the latter being a $\Pi_{2}^{0}$-complete set (\cite[Lemma 4.1]{Chiodo3}). So the set of finite presentations of $A$-torsion-free groups is $\Pi_{2}^{0}$-hard, for any non-empty $A\subseteq \mathbb{N}_{\geq 2}$. 

Moreover, when $A$ is also r.e.,  this set has the following $\Pi_{2}^{0}$ description:
\[
 G \textnormal{ is $A$-torsion-free } \Leftrightarrow (\forall w \in G)(\forall n \in A)(w^{n}\neq_{G}e \textnormal{ or } w =_{G} e)
\]
and is thus $\Pi_{2}^{0}$-complete.
\end{proof}

Note that the recursive presentation $P_{n}$  can be constructed uniformly from $n$ whenever $A$ is r.e.~and non-empty, as to find our $1<a\in A$ we simply begin an enumeration of $A$ and take the first output.

\subsection{$X$-torsion-length}\mbox{}

We finish this section by generalising the notion of torsion length, which was first introduced in \cite[Definition 2.5]{ChiVya}.

\begin{defn}\label{tordefn}
Given $X \subseteq \mathbb{N}$, we define the $X$-\emph{Torsion Length} of $G$, $\torlen^{X}(G)$, by the smallest ordinal $n$ such that $\tor^{X}_{n}(G)=\tor^{X}_{\omega}(G)$. 
\end{defn}

Rather than go and re-work all the theory developed in \cite{ChiVya}, we simply state here the main results \cite[Theorem 3.3 and Theorem 3.10]{ChiVya}, generalised to $X$-torsion. Going through the work in \cite{ChiVya}, it is straightforward to see that \emph{all} results there generalise to $X$-torsion, and thus we refrain from doing so here.

\begin{thm}\label{fp tor len}
Given any $\emptyset \neq X \subseteq \mathbb{N}_{\geq 2}$, there is a family of finite presentations $\{ P_{n}\}_{n \in \mathbb{N}}$ of groups satisfying $\torlen^{X}(\overline{P}_{n})=n$ and $\overline{P}_{n} /  \tor^{X}_{1}(\overline{P}_{n}) \cong \overline{P}_{n-1}$. 
\end{thm}

\begin{thm}\label{fp inf torlen}
Given any $\emptyset \neq X \subseteq \mathbb{N}_{\geq 2}$, there exists a $2$-generator recursive presentation $Q$ for which $\torlen^{X}(\overline{Q}) = \omega$. If $X$ is r.e., then we can algorithmically construct such a finite presentation $Q$ from $X$.
\end{thm}

\section{Presentations}\label{sec:model_proof}

The purpose of this section is to re-prove Theorem \ref{uni cbl tor free} using model-theoretic arguments. We follow the idea in \cite[Theorem A.1]{BS08}.

\subsection{Notation}\mbox{}

Throughout this section $\mathcal{L}_{\mathrm{Grp}}$ will be used to denote the language of group theory, that is the language of first-order logic supplemented with a binary function symbol $\cdot$ whose intended interpretation is the group operation, a constant symbol $e$ whose intended interpretation is the identity element, and a unary function symbol $\Box^{-1}$ whose intended interpretation is the function that sends elements to their inverses. We use the standard abbreviation of writing $x^n$ instead of $\underbrace{x \cdot \cdots \cdot x}_n$. If $X$ is a set of constant symbols that are not in $\mathcal{L}$ endowed with a well-ordering of its elements then we write $\mathcal{L}_X$ for the language obtained by adding the constant symbols in $X$ to $\mathcal{L}$. If $\mathcal{L}$ is a language, $X$ is a set of new constant symbols that are not in $\mathcal{L}$ endowed with an implicit well-ordering, $\mathcal{M}$ is an $\mathcal{L}$-structure, and $A \subseteq \mathcal{M}$ with a canonical bijection witnessing that $|A|= |X|$, then we write $\langle \mathcal{M}, A \rangle$ for the $\mathcal{L}_X$-structure obtained by interpreting the constant symbols in $X$ with the elements of $A$. If $\mathcal{M}$ is an $\mathcal{L}$-structure and $A \subseteq \mathcal{M}$, then we say that $\mathcal{M}$ is generated by $A$ if all of $\mathcal{M}$ is obtained by closing $A$ and the constants of $\mathcal{M}$ under applications of the interpretation of functions from $\mathcal{L}$ in $\mathcal{M}$.

\subsection{Model theory preliminaries}\mbox{}

We begin by recalling some definitions and results from Chapter 9 of \cite{hod93}.

\begin{defn}
Let $\mathcal{L}$ be a language. We say that an $\mathcal{L}_{\infty \infty}$-formula $\phi$ is \emph{basic Horn} if $\phi$ is in the form
$$\bigwedge \Phi \Rightarrow \psi,$$
where $\Phi$ is a, possibly infinite, set of atomic $\mathcal{L}$-formulae, and $\psi$ is either an atomic $\mathcal{L}$-formula or $\bot$. We say that an $\mathcal{L}_{\infty \infty}$-formula $\phi$ is \emph{universal Horn}, and write $\forall_1$ Horn, if $\phi$ is in the form $\forall \vec{x} \theta(\vec{x})$ where $\theta$ is basic Horn. We say that an $\mathcal{L}$-theory $T$ is universal Horn if $T$ has an axiomatisation that only consists of $\forall_1$ Horn sentences. 
\end{defn}

Let $T_{\mathrm{Grp}}$ be the obvious $\mathcal{L}_{\mathrm{Grp}}$-theory that axiomatises the class of groups. It is clear that that $T_{\mathrm{Grp}}$ can be written as a finite set of finitary $\forall_1$ Horn sentences.

\begin{defn}\label{Df:XTorsionFreeTheory}
Let $X \subseteq \mathbb{N}$. We write $T_{X-\tf}$ for the $\mathcal{L}_{\mathrm{Grp}}$-theory with axioms:
$$T_{\mathrm{Grp}} \cup\{ \forall x(x^n = e \Rightarrow x=e ) \mid n \in X \}.$$ 
\end{defn}

It is clear that for all $X \subseteq \mathbb{N}$, $T_{X-\tf}$ is a finitary universal Horn theory that axiomatises the class of $X$-torsion-free groups.  

\begin{defn}
Let $\mathcal{L}$ be a language and let $\mathbb{K}$ be a class of $\mathcal{L}$-structures. A \emph{$\mathbb{K}$-presentation} is a tuple $\langle X, \Phi \rangle_{\mathbb{K}}$ such that $X$ is a set of new constant symbols endowed with an implicit well-ordering, called generators, that are not in $\mathcal{L}$, and $\Phi$ is a set of atomic $\mathcal{L}_X$-sentences. We will write $\langle X, \Phi \rangle$ instead of $\langle X, \Phi \rangle_{\mathbb{K}}$ when $\mathbb{K}$ is clear from the context. If both $\Phi$ and $X$ are finite then we say that $\langle X, \Phi \rangle_{\mathbb{K}}$ is a \emph{finitely presented} $\mathbb{K}$-presentation. If both $\Phi$ and $X$ are r.e. then we say that $\langle X, \Phi \rangle_{\mathbb{K}}$ is a \emph{recursively presented} $\mathbb{K}$-presentation.    
\end{defn}

\begin{defn}
Let $\mathcal{L}$ be a language and let $\mathbb{K}$ be a class of $\mathcal{L}$-structures. Let $\langle X, \Phi \rangle_{\mathbb{K}}$ be a $\mathbb{K}$-presentation. We say that an $\mathcal{L}_X$-structure $\langle \mathcal{M}, A \rangle$ is a \emph{model} of $\langle X, \Phi \rangle_{\mathbb{K}}$ if 
$$\mathcal{M} \in \mathbb{K} \textrm{ and } \langle \mathcal{M}, A \rangle \models \bigwedge \Phi.$$
\end{defn}

\begin{defn} \label{Df:Presents}
Let $\mathcal{L}$ be a language and let $\mathbb{K}$ be a class of $\mathcal{L}$-structures. Let $\langle X, \Phi \rangle_{\mathbb{K}}$ be a $\mathbb{K}$-presentation. We say that $\langle X, \Phi \rangle_{\mathbb{K}}$ \emph{presents} an $\mathcal{L}_X$-structure $\langle \mathcal{M}, A \rangle$ if 
\begin{itemize}
\item[(i)] $\langle \mathcal{M}, A \rangle$ is a model of $\langle X, \Phi \rangle_{\mathbb{K}}$,
\item[(ii)] $\mathcal{M}$ is generated by $A$,
\item[(iii)] for every model $\langle \mathcal{N}, B \rangle$ of $\langle X, \Phi \rangle_{\mathbb{K}}$, there exists a homomorphism $f: \mathcal{M} \longrightarrow \mathcal{N}$ such that $f(c^\mathcal{M})= c^\mathcal{N}$ for all $c \in X$. 
\end{itemize}
We say that $\mathbb{K}$ \emph{admits presentations} if every $\mathbb{K}$-presentation presents a structure $\langle \mathcal{M}, A \rangle$ with $\mathcal{M} \in \mathbb{K}$.   
\end{defn}

Note that if $\langle X, \Phi \rangle$ presents $\langle \mathcal{M}, A \rangle$ and $\langle \mathcal{N}, B \rangle$ is a model of $\langle X, \Phi \rangle$, then, since $\mathcal{M}$ is generated by $A$, the homomorphism whose existence is guaranteed by Definition \ref{Df:Presents}(iii) is unique.

The following is Lemma 9.2.1 of \cite{hod93}:

\begin{lem} \label{Th:ConditionForPresents}
Let $\Phi$ be a set of atomic $\mathcal{L}_X$-sentences where $X$ is a set of constants not in $\mathcal{L}$. Let $\langle X, \Phi \rangle$ be a $\mathbb{K}$-presentation and let $\langle \mathcal{M}, A \rangle$ be an $\mathcal{L}_X$-structure with $\mathcal{M} \in \mathbb{K}$. The following are equivalent:
\begin{itemize}
\item[(i)] $\langle X, \Phi \rangle$ presents $\langle \mathcal{M}, A \rangle$, 
\item[(ii)] $A$ generates $\mathcal{M}$; and for every atomic formula $\psi(\vec{x})$ of $\mathcal{L}$ and for every $\vec{a} \in A$,
$$\mathcal{M} \models \psi(\vec{a}) \textrm{ if and only if every structure in } \mathbb{K} \textrm{ is a model of } \forall \vec{x} \left( \bigwedge \Phi \Rightarrow \psi \right).$$
\end{itemize}
\end{lem}

And this is Lemma 9.2.2 of \cite{hod93}:

\begin{lem} \label{Th:ConditionAdmitsPresentations}
Let $\mathcal{L}$ be a language and let $\mathbb{K}$ be a class of $\mathcal{L}$-structures which is closed under isomorphic copies. The following are equivalent:
\begin{itemize}
\item[(i)] $\mathbb{K}$ is closed under products, $\mathbf{1}$ and substructures,
\item[(ii)] $\mathbb{K}$ admits presentations,
\item[(iii)] $\mathbb{K}$ is axiomatised by a universal Horn theory in the language $\mathcal{L}_{\infty \infty}$.
\end{itemize}
\end{lem}

\subsection{A proof of Theorem \ref{uni cbl tor free} using model theory}\mbox{}

The following is an adaptation of the proof  of Theorem A.1 in  \cite{BS08}.

\begin{thm} \label{Th:GeneralUniversalGroupExistenceTheorem}
If $T \supseteq T_{\mathrm{Grp}}$ is an r.e.~universal Horn $\mathcal{L}_{\mathrm{Grp}}$-theory then there exists a recursively presented group $G \models T$ such that every recursively presented group $H \models T$ embeds into $G$.  
\end{thm}

\begin{rem} A group property $\rho$ having an r.e.~universal Horn theory does \emph{not} imply that the finite presentations of groups with $\rho$ are r.e. Indeed, for $X$ a non-empty r.e.~set, $T_{X-\tf}$ is an r.e.~universal Horn theory (Definition \ref{Df:XTorsionFreeTheory}), but by Theorem \ref{pi2} the set of finite presentations of such groups is $\Pi^{0}_{2}$-complete.
\end{rem}

\begin{proof}
Let $T \supseteq T_{\mathrm{Grp}}$ be a r.e.~universal Horn $\mathcal{L}_{\mathrm{Grp}}$-theory. Let $\mathbb{K}$ be the class of $\mathcal{L}_{\mathrm{Grp}}$-structures that satisfy $T$. Let $\mathbb{K}^\prime$ be the class of $\mathcal{L}_{\mathrm{Grp}}$-structures that satisfy $T_{\mathrm{Grp}}$. By Lemma \ref{Th:ConditionAdmitsPresentations}, both $\mathbb{K}$ and $\mathbb{K}^\prime$ admit presentations. If $\tau$ is a presentation, then we will write $G^\tau$ for the element of $\mathbb{K}$ presented by $\tau$, and $G_\tau$ for the element of $\mathbb{K}^\prime$ presented by $\tau$. Let $\langle \pi_n \mid n \in \mathbb{N} \rangle$ be effective enumeration of recursive presentations. Let $\pi$ be the disjoint union of all the $\pi_n$'s; $\pi:=\pi_{1}*\pi_{2}*\cdots$. Therefore, $\pi$ is a recursive presentation. We claim that $G^\pi$ is the desired universal group satisfying $T$. It is immediate that $G^\pi \models T$. If $\tau$ is a recursive presentation such that $G_\tau \models T$ then $G_\tau= G^\tau$, and so $G_\tau$ embeds into $G^\pi$. This shows that $G^\pi$ is universal. It remains to show that $G^\pi$ is recursively presented in $\mathbb{K}^\prime$. Let $\pi= \langle Y, \Phi \rangle$ and let $Y^{G^\pi}$ be the interpretations of the constant symbols $Y$ in $G^\pi$. Let $\Phi^\prime$ be the set of atomic $\mathcal{L}_{\mathrm{Grp}, Y}$-sentences (the language obtained by adding new constants for the generators in $Y$) that hold in $\langle G^\pi, Y^{G^\pi} \rangle$. Lemma \ref{Th:ConditionForPresents} implies that $\langle Y, \Phi^\prime \rangle_{\mathbb{K}^\prime}$ presents $\langle G^\pi, Y^{G^\pi} \rangle$. We need to show is that $\Phi^\prime$ is r.e. Let $S= T \cup \Phi$. We claim that for all atomic $\mathcal{L}_{\mathrm{Grp}, Y}$-sentences $\sigma$,
\begin{equation} \label{Eq:Equation1}
S \vdash \sigma \textrm{ if and only if } \sigma \in \Phi^\prime.
\end{equation}
Since $G^\pi \models S$, it follows that for all atomic $\mathcal{L}_{\mathrm{Grp}, Y}$-sentences $\sigma$, if $S \vdash \sigma$ then $\sigma \in \Phi^\prime$. We need to show the converse. Let $\sigma \in \Phi^\prime$, and suppose that $S \nvdash \sigma$. Note that, since $\sigma$ is atomic, it is of the form $w_1= w_2$ where $w_1$ and $w_2$ are words in the generators $Y$. Let $H$ be a group such that $Y^H$ is the interpretation of the constant symbols $Y$ in $H$, $\langle H, Y^H \rangle \models S$, and $H \models (w_1 \neq w_2)$. Therefore $\langle H, Y^H \rangle$ is a model of $\langle Y, \Phi \rangle_{\mathbb{K}}$. It follows from (iii) of Definition \ref{Df:Presents} that the map $g: Y^{G^\pi} \longrightarrow Y^H$ defined by the implicit ordering on $Y$ must lift to a homomorphism $f: G^\pi \longrightarrow H$. But this map is not well-defined since $w_1=w_2$ in $G^\pi$ and $f(w_1) \neq f(w_2)$ in $H$, which is a contradiction. Therefore (\ref{Eq:Equation1}) holds. Since $S$ is r.e., (\ref{Eq:Equation1}) yields a recursive enumeration of the elements of $\Phi^\prime$. 
\end{proof}

If $X \subseteq \mathbb{N}$ is r.e.~then Definition \ref{Df:XTorsionFreeTheory} shows that $T_{X-\tf}$ is an r.e.~universal Horn theory that extends $T_{\mathrm{Grp}}$. This immediately re-proves Theorem \ref{uni cbl tor free}. 
It is then an application of Theorem \ref{tor emb} again, to re-prove Theorem \ref{uni tor free}. This time, we have done most of the work using model-theoretic techniques and in a way which could be applied  to other  group properties (not just $X$-torsion-freeness), as discussed in the introduction.

\ 

\noindent \scriptsize{\textsc{Department of Pure Mathematics and Mathematical Statistics, University of Cambridge,
\\Wilberforce Road, Cambridge, CB3 0WB, UK. 
\\mcc56@cam.ac.uk
\vspace{5pt}
\\Department of Philosophy, Linguistics and Theory of Science, Gothenburg University,
\\Olof Wijksgatan 6, 405 30 Gothenburg, Sweden. 
\\zachiri.mckenzie@gu.se}

\end{document}